\documentclass{article}
\usepackage[utf8]{inputenc}
\usepackage[english]{babel}
\usepackage{amsmath} 
\usepackage{dsfont}
\usepackage{centernot}
\usepackage{amssymb}
\usepackage[a4paper,top=2.54cm,bottom=2.0cm,left=2.0cm,right=2.4cm]{geometry}
\usepackage{mathtools}
\usepackage{amsthm}
\usepackage{mathrsfs}
\usepackage{enumitem}
\usepackage{xcolor}
\usepackage{hyperref}
\usepackage{bbm}
\usepackage{comment}
\usepackage{indentfirst}
\renewcommand{\subset}{\subseteq}
\usepackage{authblk} %to put several authors in a nice way
\usepackage{tikz-cd}
\usetikzlibrary{decorations.markings}
\tikzset{degil/.style={
            decoration={markings,
            mark= at position 0.5 with {
                  \node[transform shape] (tempnode) {$/$};
                  }
              },
              postaction={decorate}
}
}

\setlist*[enumerate]{label={\normalfont(\arabic*)}}

\DeclareMathOperator{\cl}{cl}

\DeclareMathOperator{\inter}{int}

\newcommand{\A}{\mathcal{A}}

\makeatletter
\def\moverlay{\mathpalette\mov@rlay}
\def\mov@rlay#1#2{\leavevmode\vtop{%
   \baselineskip\z@skip \lineskiplimit-\maxdimen
   \ialign{\hfil$\m@th#1##$\hfil\cr#2\crcr}}}
\newcommand{\charfusion}[3][\mathord]{
    #1{\ifx#1\mathop\vphantom{#2}\fi
        \mathpalette\mov@rlay{#2\cr#3}
      }
    \ifx#1\mathop\expandafter\displaylimits\fi}
\makeatother

\numberwithin{equation}{section}
\theoremstyle{plain}
\newtheorem{teo}[equation]{Theorem}
\newtheorem{prop}[equation]{Proposition}

\newtheorem{cor}[equation]{Corollary}
\newtheorem{defin}[equation]{Definition}
\newtheorem{question}[equation]{Question}

\theoremstyle{remark}

\theoremstyle{definition}

\title{Special sets of reals and weak forms of normality \\on Isbell-Mrówka spaces}
\author[1]{Vinicius de Oliveira Rodrigues%
\thanks{Electronic address: \texttt{vinior@ime.usp.br}; Corresponding author}}

\author[1]{Victor dos Santos Ronchim%
 \thanks{Electronic address: \texttt{vronchim@ime.usp.br}}}
 
\author[2]{Paul Szeptycki%
\thanks{Electronic address: \texttt{szeptyck@yorku.ca}}}
\affil[1]{Institute of Mathematics and Statistics, University of São Paulo}
\affil[2]{Department of Mathematics and Statistics, York University}

\date{\today}

\begin{document}
%\tableofcontents
\newpage

\maketitle
\begin{abstract}
We recall some classical results relating normality and some natural weakenings of normality in $\Psi$-spaces over almost disjoint families of branches in the Cantor tree to special sets of reals like $Q$-sets, $\lambda$-sets and $\sigma$-sets. We introduce a new class of special sets of reals which corresponds the corresponding almost disjoint family of branches being $\aleph_0$-separated. This new class fits between $\lambda$-sets and perfectly meager sets. We also discuss conditions for an almost disjoint family $\mathcal A$ being potentially almost-normal (pseudonormal), in the sense that $\mathcal A$ is almost-normal (pseudonormal) in some c.c.c. forcing extension.
\vspace{1em}
    
    

\end{abstract}

\section{Introduction}
Given a countable infinite set $N$, an almost disjoint family (on $N$) is a infinite collection of infinite subsets of $N$ whose pairwise intersections are finite. Throughout this paper we assume that $[N]^\omega\cap N=\emptyset$. A MAD family (maximal almost disjoint family) is an almost disjoint family which is not properly contained in any other family. It is well known that there are almost disjoint families of size $\mathfrak c$ \cite{blass2010}.

Each almost disjoint family $\mathcal A$ on $N$ is associated to the \textit{Isbell-Mrówka} space of $\mathcal A$, also called the $Psi$-space of $\mathcal A$ and denoted by $\Psi(\mathcal A)$. This space is the set $N\cup \mathcal A$, where $N$ is open and discrete and for each $a \in \mathcal A$, the sets of the form $\{a\}\cup(a\setminus F)$, where $F\subseteq N$ is finite, form an open basis for the point $a$. It is easy to verify that $\mathcal A$ is a Hausdorff, locally compact (and therefore Tychonoff) zero-dimensional non-compact separable Moore topological space.

 The topological properties of $\Psi(\mathcal A)$ often depend on the combinatorial properties of $\mathcal A$. For instance, $\Psi(\mathcal A)$ is pseudocompact iff $\mathcal A$ is MAD and $\Psi(\mathcal A)$ is metrizable iff $\mathcal A$ is countable. $\Psi$-spaces provide a rich source of counter-examples to many topoligical questions including questions about convergence and compactness, and there are many interesting and difficult problems about the combinatorial properties of almost disjoint families that are interesting in their own right. As good introductions to this field of study we mention \cite{Hrusak2014} and \cite{Hernandez-Hernandez2018}.
 
 To improve our notation, we say that an almost disjoint family $\mathcal A$ has a certain topological property if and only if $\Psi(\mathcal A)$ has this topological property. So the statement ``$\mathcal A$ is normal'' means the same as ``$\Psi(\mathcal A)$ is normal''.
 
 The normality of Isbell-Mrówka spaces has been extensively studied, perhaps initially in relation to the Normal Moore Space Problem. These problems are closely related to the concept of $Q$-set, a special kind of sets of reals:
 
 \begin{defin}
 Let $X$ be a Polish space. We say $A\subseteq X$ is a $Q$-set iff $A$ is uncountable and every subset of $A$ is a $G_\delta$ relatively to $A$. We say $A\subseteq X$ is a $\lambda$-set iff $A$ is uncountable and every countable subset of $A$ is a $G_\delta$ relatively to $A$.
 \end{defin}
 
We have the following classical result. We give \cite{qsetsnormality} as a reference.
\begin{teo} The following are equivalent:
\begin{enumerate}
    \item There is a $Q$-set.
    \item There is an uncountable normal Isbell-Mrówka space.
    \item There is a separable normal non-metrizable Moore space.
\end{enumerate}
\end{teo}

The proof (1)$\leftrightarrow$(2) of the proposition above is sketched as follows: given a set $X\subseteq 2^\omega$ of reals, we may define an almost disjoint family on $N=2^{<\omega}$ named $\mathcal A_X$ consisting of the sets $a_x=\{x|_n: n \in \omega\}$ for $x \in X$. Then it is shown that $\mathcal A_X$ is normal iff $X$ is a $Q$-set. Moreover, it is shown that if $\mathcal A$ is a normal almost disjoint family on $\omega$, then it is a $Q$-set of the Polish space $\mathcal P(\omega)$.

Of course, the existence of a $Q$-set (and so all the above statements) are independent of the axioms of ZFC. For instance, under CH (2) fails since such a space would be a separable normal space with a closed discrete subspace of size $\mathfrak c$, which violates Jones's Lemma, and under $\mathfrak p=\mathfrak c$, every set of reals of size $<\mathfrak c=\mathfrak p$ is a $Q$-set (since the natural poset for making a subset of a given set a relative $G_\delta$ is $\sigma$-centered, as mentioned in \cite{Brendle}).

The main contribution of this paper is the introduction of a new class of special subsets of the reals which we call \textit{weak} $\lambda$-set. For information on the classical special subsets of reals, such as $\sigma$-sets, $\lambda$-sets and $Q$-sets we refer to \cite{miller1984special}.

In the past decades several weakenings of normality have been proposed and studied. A topological space is pseudonormal if for every closed set $K$ and every countable closed set $F$ there exist disjoint open sets $U, V$ such that $F\subseteq U$, $K\subseteq V$. It is mentioned in \cite{qsetsnormality}, while a $Q$-set $X\subseteq 2^\omega$ gives a normal $\Psi$-space, we have $X\subseteq 2^\omega$ is a $\lambda$-set if and only if $\mathcal A_X$ is pseudonormal. Conversely, if $\mathcal A$ is a pseudonormal almost disjoint family, then it is a $\lambda$-set of $\mathcal P(\omega)$. The proofs are similar.

Other weakenings of normality have been studied in the realm of Isbell-Mrówka spaces. We cite \cite{paulsergio}, \cite{vinichim} as recent works.

In \cite{paulsergio}, several weakenings of normality in Isbell-Mrówka spaces have been studied including the notion of almost-normal.  We say that a topological space is almost-normal iff for every closed set $K$ and every regular closed set $F$ disjoint from $K$ there exists two disjoint open sets $U$, $V$ such that $K\subseteq U$ and $F\subseteq V$. They asked if there exists an almost normal Isbell-Mrówka space which is not normal, or, more strongly, if there exists a MAD family whose Isbell-Mrówka space is almost-normal. In \cite{vinichim}, V. Rodrigues and V. Ronchim used forcing to show that the answer to the former is consistently positive with both CH and $\neg$CH. The latter is still open.

In \cite{vinichim}, in order to produce the mentioned example something which was defined as ``almost $Q$-set'' was used. This definition, which will be stated in the next section, was designed to work with almost-normality in the same way as $Q$-sets work with normality, that is, in a way such that $X\subseteq 2^{\omega}$ is an almost $Q$-set if and only if $\mathcal A_X$ is almost-normal. In this paper, we show that this class of sets actually is the well known class of the $\sigma$-sets of reals (see \cite{miller1984special}).

In \cite{paulsergio}, P. Szeptycki and S. Garcia-Balan defined the notion of an strongly $\aleph_0$-separated almost disjoint family, which is an almost disjoint family where every pair of disjoint closed countable sets can be separated by a clopen set. They proved that every almost-normal almost disjoint family is strongly $\aleph_0$-separated. In this paper we introduce a new class of special subsets of the reals which we call weak $\lambda$-sets which fits between the class of $\lambda$-sets and the class of perfectly meager sets ($X\subseteq 2^\omega$ is perfectly meager if and only if its intersection with every perfect set $P$ is meager in $P$). It follows that $X\subseteq 2^\omega$ is weak $\lambda$ if and only if $\mathcal A_X$ is strongly $\aleph_0$-separated, and that if an almost disjoint family $\mathcal A$ is strongly $\aleph_0$-separated then it is a weak $\lambda$-set of $\mathcal P(\omega)$  Consistently, there are perfectly meager sets which are not weak $\lambda$-sets. We do not know if there is a weak $\lambda$-set which is not a $\lambda$-set.
\section{Almost-normality and \texorpdfstring{$\sigma$}{sigma}-sets}
As mentioned in the introduction, throughout this paper we will give special attention to the almost disjoint families of branches in the Cantor tree $2^{<\omega}$:

\begin{defin}
Let $X\subseteq 2^\omega$, the almost disjoint family induced by $X$ (over $N=2^{<\omega}$) is the family $\A_X=\{ a_x: x\in X\}$, where $a_x=\{ x|_n: n\in\omega\}$ for each $x\in X$.

We denote $\widehat{X} = \{x|_n : n\in\omega, x\in X\}$ and, for a subset $K\subseteq \Psi(\mathcal A_X)$, we define $\langle K\rangle_X = \{x\in X: a_x\in K\}$. 
\end{defin}

Notice that almost disjoint families of branches are never MAD because each element of $\mathcal A_X$ intersects each infinite anti-chain in $2^{<\omega}$ at most in one point. It is worth mentioning that for this special class of almost disjoint families  some topological properties from $\Psi(\mathcal A_X)$ can be characterized in terms of topological properties the set $X$. The proposition below is probably folklore, but we include the proof for completeness.

\begin{prop}[Folklore]\label{prop: fsigma sse separa}
Given $X\subseteq 2^\omega$ and $Y\subseteq X$. The following are equivalent:
\begin{enumerate}
    \item $\mathcal A_Y$ and $\mathcal A_{X\setminus Y}$ can be separated in $\A_X$;
    \item $Y$ and $X\setminus Y$ are $F_\sigma$ in $X$.
\end{enumerate}
\end{prop}
\begin{proof}
$(1)\implies(2)$: Let $Z\subseteq 2^{<\omega}$ a partitioner for $\mathcal A_Y$ and $\mathcal A_{X\setminus Y}$ such that for all $y\in Y$ and $x\in X\setminus Y$, $a_y\subseteq^* Z$ and $a_x\cap Z=^*\emptyset$. It follows that:
\[
Y= \{x\in X: a_y\subseteq^* Z\} = \bigcup_{n\in\omega} \bigcap_{m\geq n} \underbrace{ \{y\in X: y|_m\in Z\} }_{\text{closed in $2^\omega$ }}.
\]

Notice that $Z_0 = 2^{<\omega}\setminus Z$ is a partitioner for $\mathcal A_Y$ and $\mathcal A_{X\setminus Y}$ such that $A_x\subseteq^* Z_0$ iff $x\in X\setminus Y$, one concludes that $X\setminus Y$ is also an $F_\sigma$ set of $X$.

$(2)\implies(1)$: It follows from the argument in \cite[Lemma 4.3]{vinichim}.
\end{proof}

Given an almost disjoint family $\mathcal A$ over $N$, we say that a set $J\subseteq N$ is a partitioner for disjoint subfamilies $\mathcal B,\mathcal C\subseteq \mathcal A$ iff the following are satisfied:
\begin{enumerate}
    \item For all $a\in\mathcal A$: $a\subseteq^* J$ or $a\cap J=^*\emptyset$;
    \item For all $b\in \mathcal B$, $b\subseteq^* J$;
    \item For all $c\in\mathcal C$, $c\cap J =^*\emptyset$.
\end{enumerate}

It is well known and easy to see that an almost disjoint family $\mathcal A$ is normal iff for all $\mathcal B\subseteq \mathcal A$, there exist a partitioner for $\mathcal B$ and $\mathcal A\setminus\mathcal B$. As an immediate consequence of this fact and the previous result, we obtain the following folklore result:

\begin{cor}
$X$ is a $Q$-set of $2^\omega$ if, and only if, $\Psi(\mathcal A_X)$ is normal.
\end{cor}

Recall that an almost disjoint family $\mathcal A$ is strongly $\aleph_0$-separated iff each pair of countable disjoint subfamilies $\mathcal B,\mathcal C \subseteq \mathcal A$ can be separated by a partitioner. This definition was first presented in $\cite{paulsergio}$, it is weaker than almost-normality and under CH there exist strongly $\aleph_0$-separated MAD families.

\begin{cor}\label{cor:FsigmaGdelta}
Let $X\subset 2^\omega$. Suppose that, for every pair of countable disjoint subsets $Y,Z\subset X$ there exists a $F_\sigma$-$G_\delta$ set $F\subset 2^\omega$ satisfying:
\[
\mathcal A_Y\subset \mathcal A_{F\cap X} \qquad \text{ and } \qquad \mathcal A_Z\subset \mathcal A_{X\setminus F},
\]
then $\Psi(\A_X)$ is strongly $\aleph_0$-separated.
\end{cor}
\begin{proof}
Let $\mathcal B$ and $\mathcal C$ countable disjoint subfamilies of $\A_X$. Then there exist countable disjoint subsets $Y,Z\subset X$ such that $\mathcal B = \mathcal A_Y$ and $\mathcal C = \mathcal A_Z$. Let $F \subset X$ a $F_\sigma$-$G_\delta$ set that separates $\mathcal B$ and $\mathcal C$. By Proposition~\ref{prop: fsigma sse separa}, $\mathcal B$ and $\A_X\setminus \mathcal B$ are separated. Thus, $\A_X$ is strongly $\aleph_0$-separated.
\end{proof}

\begin{prop}\label{prop: gdelta reg closed}
Let $K\subset \Psi(\A_X)$. The following are equivalent:
\begin{enumerate}
    \item There exists $W\subset 2^{<\omega}$ such that $K = \cl(W)\cap\A_X$;
    \item $\langle K\rangle_X = \{x\in X: a_x\in K\}$ is $G_\delta$ in $X$.
\end{enumerate}
\end{prop}
\begin{proof}

$(1)\implies (2)$: Let $W\subset 2^{<\omega}$ such that $K=\cl(W)\cap\A_X$. It follows that:
\[
\{x\in X: a_x\in K \} = \{x\in X: |a_x\cap W|=\omega\} = \bigcap_{n\in\omega} \bigcup_{m\geq m} 
\underbrace{\{ x\in X: x|_m\in W\}}_{\text{open set in $X$}}.
\]
Thus, it is a $G_\delta$-set of $X$.

$(2)\implies(1)$: Suppose $\langle K\rangle_X$ is a $G_\delta$ of $X$. Write $\langle K\rangle_X=\bigcap\limits_{n \in \omega} U_n$, where each $U_n$ is an open subset of $X$ and $U_n\subseteq U_m$ whenever $n\geq m$.

For each $n$, write $U_n=\bigcup\{[s]: s \in L_n\}$, where $L_n$ is a countable subset of $2^{<\omega}$ such that for all $s, t \in L_n$, $s, t$ are incompatible and $|s|, |t|>n$. 

Let $W=\bigcup\{L_n:  n \in \omega\}$. We claim that $\cl W\cap \mathcal A_X=K$.

Suppose $a \in \cl (W)\cap \mathcal A_X$. Let $x \in X$ be such that $a=a_x$. It suffices to see that $x \in U_n$ for every $n \in \omega$. Fix $n$. Since $a_x \in \cl (W)$, there exists infinitely many $m \in \omega$ such that $x|_m \in W=\bigcup_{k \in \omega} L_k$. Since all members of $L_k$ are pairwise incompatible, for each k, $x|_m \in L_k$ for at most one $m$. So there exists $m \in \omega$ and $k\geq n$ such that $x|_m \in L_k$, so $x \in U_k\subseteq U_n$.

On the other hand, if $a \in K$, let $x \in \langle K\rangle_X$ be such that $a=a_x$. Then $x \in U_n$ for all $n \in \omega$, that is, for each $n \in \omega$ there exists $s_n \in L_n$ such that $s_n \subseteq x$. Since $|s_n|>n$ for each $n$ and $s_n \in W$, this implies that $x \in \cl (W)$.
\end{proof}

In \cite[Theorem 3.6]{vinichim} equivalent conditions to almost-normality in $\Psi$-spaces were presented. In particular, the authors have shown that for an arbitrary almost disjoint family $\mathcal A$, the following holds:
\begin{equation}\label{eq: equivalences of almost normality}
\Psi(\mathcal A) \text{ is almost normal} \iff 
\begin{tabular}{l}
 For all regularly closed set $F$, there exists\\ a partitioner 
   for $F\cap\mathcal A$ and $\mathcal A\setminus F$
\end{tabular}
\end{equation}

\begin{defin}
Let $X$ be a Polish space. We say that a subset $A\subseteq X$ is a $\sigma$-set iff every relative $G_\delta$ subset of $A$ is a relative $F_\sigma$.
\end{defin}
\begin{defin}
An almost $Q$-set in $2^\omega$ is an uncountable subset $X\subseteq 2^\omega$ such that for every $W \subseteq 2^{<\omega}$, $[W]_X=\{x \in X: \forall m\in \omega\,  \exists n\geq m\,( x|_n \in W)\}$ (which is $\{x \in X: |a_x\cap W|=\omega\}$) is an $F_\sigma$ in $X$.
\end{defin}

Combining (\ref{eq: equivalences of almost normality}) and Proposition~\ref{prop: gdelta reg closed} yields an indirect proof that the almost $Q$-sets defined in \cite{vinichim} are, in fact, the well-known class of $\sigma$-sets:

\begin{teo}\label{teo: sigma if and only if almost-normal}
Let $X\subset 2^\omega$. The following are equivalent:
\begin{enumerate}
    \item $\A_X$ is almost-normal;
    \item $X$ is $\sigma$-set;
    \item $X$ is an almost $Q$-set.
\end{enumerate}
\end{teo}

\begin{proof}
The equivalence between $(1)$ and $(3)$ is established in \cite[Corollary 4.4]{vinichim}.

$(1)\Longrightarrow(2)$ Let $Y\subset X$ a relative $G_\delta$ set and consider $K=\mathcal A_Y$. Notice that $\langle K\rangle_X=Y$ is $G_\delta$ set, thus by Proposition~\ref{prop: gdelta reg closed} there exists $W\subset 2^{<\omega}$ such that $\mathcal A_Y=\cl(W)\cap \A_X$. By (\ref{eq: equivalences of almost normality}) there exists a partitioner $Z\subset 2^\omega$ for $\mathcal A_Y$ and $\mathcal A_{X\setminus Y}$. Then, by Proposition~\ref{prop: fsigma sse separa}, $Y$ is also a relative $F_\sigma$ set.

$(2)\Longrightarrow(3)$ It is clear that for every $W\subseteq 2^{<\omega}$, $[W]_X$ is a $G_\delta$.
\end{proof}

Recall that a Luzin family (Luzin$^*$ family) is an almost-disjoint family $\mathcal A$ of size $\omega_1$ for which there exists an injective enumeration $\mathcal A=\{a_\alpha: \alpha<\omega_1\}$ such that $\forall \alpha<\omega_1 \forall n \in \omega\, \{\beta<\alpha: a_\beta\cap a_\alpha \subseteq n\}$ is finite (such that $\forall \alpha<\omega_1 \forall n \in \omega\, \{\beta<\alpha: |a_\beta\cap a_\alpha| < n\}$ is finite). Clearly, every Luzin$^*$ family is a Luzin family. Luzin families fail to be normal badly: for every pair of disjoint uncountable sets $\mathcal B$, $\mathcal C$ of $\mathcal A$ there is no $X\subseteq \omega$ such that for every $a \in \mathcal B$, $a\subseteq^*X$ and $\forall a \in \mathcal C$, $a\cap X=^*\emptyset$.

In \cite{nluzin}, Hru\v s\'ak and Guzmán introduced the notion of an almost disjoint family potentially having a property $P$. Given a property $P$ of almost disjoint families and an almost disjoint family $\mathcal A$, they defined \textit{$\mathcal A$ is potentially $P$} if there exists a ccc forcing notion $\mathbb P$ such that $\mathbbm 1\Vdash_{\mathbb P} \check{\mathcal A} \text{ is } P$. They showed that $\mathcal A$ is potentially normal iff $\mathcal A$ has no $n$-Luzin gap (see their paper for the definition).

%\begin{defin}
%\textup{\cite[Definition 0.5]{nluzin}} \label{def: nLuzin gap}
%Let $n\in \omega$ and $B_i = \{B_{\alpha}^i: \alpha \in\omega_1\}$ be disjoint subfamilies of an AD family $\mathcal A$ for $i < n$. We call $(B_i)_{i < n}$ an
%$n$-Luzin gap if there exists a $m \in\omega$ such that:
%\begin{enumerate}
%    \item For all $i\neq j$ and $\alpha<\omega_1$: \quad $B_{\alpha}^i\cap B_{\alpha}^j \subseteq m$;
%    \item For all $\alpha, \beta <\omega_1$, $\alpha\neq\beta$: \quad$\bigcup\limits_{i\neq j}B_{\alpha}^i\cap B_{\beta}^j \nsubseteq m$.
%\end{enumerate}
%\end{defin}

We can ask if there is a nice characterization for potentially almost-normal almost disjoint families. We don't have a answer for this question. However, we have the following:

\begin{prop}[ZFC]\label{weaklyseparatedluzin}
There exists a Luzin$^*$ family which is not potentially almost-normal.
\end{prop}
\begin{proof}
Let $\{A_n: n \in \omega\}$ be a partition of $\omega$ into infinite sets. For each $n$, let $X_n$ be an infinite subset of $A_{2n}$ such that $A_{2n}\setminus A_n$ is infinite. Let $X=\bigcup_{n \in \omega}X_n$. For each infinite countable ordinal $\alpha$, let $\phi_\alpha:\omega\rightarrow \alpha$ be a bijection.

We will inductively define $(a_\alpha: \alpha<\omega)$ such that for all $\alpha<\omega_1$:

\begin{enumerate}[label=(\roman*)]
    \item $a_\alpha \in [\omega]^\omega$ and $a_\alpha\cap a_\beta$ is finite for every $\beta<\alpha$,
    \item $\forall n \in \omega \,\{\beta<\alpha: |a_\beta\cap a_\alpha|<n\}$ is finite,
    \item if $\alpha$ is odd, then $a_\alpha\cap X=\emptyset$, and
    \item if $\alpha$ is even, then $X$ splits $a_\alpha$, that is, both $a_\alpha\setminus X$ and $a_\alpha\cap X$ are infinite.
\end{enumerate}

The items (i) and (ii) guarantees that $\mathcal A=\{a_\alpha: \alpha<\omega_1\}$ is a Luzin$^*$ family, (iii) and (iv) guarantees that $X$ is such that $\{\alpha<\omega_1: |a_\alpha\cap X|=\omega\}$ is the set of even countable ordinals.

Notice that (i)-(iv) hold for $\alpha \in \omega$. Having constructed $a_\beta$ for $\beta<\alpha$ for some infinite $\alpha<\omega_1$, we construct $a_\alpha$ as follows: for each $n$, let $s_n\subseteq a_{\phi_\alpha(n)}\setminus \bigcup_{i<n}a_{\phi_\alpha(i)}$ such that $|s_n|=n$. If $\alpha$ is odd, we choose $s_n$ such that $s_n\cap X=\emptyset$, which is possible by (iii) and (iv). If $\alpha$ is even and $\phi_\alpha(n)$ is even, we choose $s_n\subseteq X$, which is possible by (iv), and if $\phi_\alpha(n)$ is odd, we choose $s_n$ such that $X\cap s_n=\emptyset$, which is possible by (iii). It is clear that by letting $a_\alpha=\bigcup\{s_n: n \in \omega\}$, (i)-(iv) are satisfied.

We claim that $\mathcal A$ is not potentially almost-normal: if $V[G]$ is a ccc forcing extension of $V$, $\mathcal A$ is still a Luzin$^*$ family in $V[G]$ (since c.c.c. forcings preserve cardinals) and $\{a_\alpha: \alpha<\omega \text{ is even}\}\cup X$ is a regular closed subset of $\Psi(\mathcal A)$ that cannot be separated from the closed set $\{a_\alpha: \alpha<\omega \text{ is odd}\}$ since that would imply the existence of a partitioner for the uncountable sets $\{a_\alpha: \alpha<\omega \text{ is even}\}\cup X$ and $\{a_\alpha: \alpha<\omega \text{ is odd}\}$, violating the fact that $\mathcal A$ is Luzin$^*$.
\end{proof}

Such a set $X$ does not exist for every Luzin family. For instance, in Example 2.10 \cite{paulsergio}, CH is used to construct a MAD Luzin family $\mathcal A$ for which for every $X\subseteq \omega$, $\{a \in \mathcal A: |a\cap X|=\omega\}$ is either finite or co-countable. It is not clear for us if that Luzin family is potentially almost-normal.

\begin{question}
Is it consistent that there is an almost-normal Luzin-family? What about a potentially almost-normal one?
\end{question}

\begin{question}
What is a nice characterization of potentially almost-normal almost disjoint families?
\end{question}

We note that for any Luzin family $\mathcal A$ and for any uncountable set $\mathcal B\subseteq \mathcal A$ whose complement is also uncountable, we can add, by a c.c.c. forcing, a set $X$ such that $\mathcal B=\{a \in A: |a\cap X|=\omega\}$, thus:

\begin{prop}
Every Luzin family is potentially not almost-normal.
\end{prop}
\begin{proof}
Let $\mathcal A$ be a Luzin family, let $\mathcal B\subseteq \mathcal A$ be a an uncountable set whose complement in $\mathcal A$ is also uncountable. Consider Solovay's poset for adding a set $X$ almost disjoint with $\mathcal A\setminus \mathcal B$, i.e., $\mathbb P=[\omega]^{<\omega}\times [\mathcal A\setminus \mathcal B]^{<\omega}$ ordered by $(s, A)\leq (s', A')$ ($\leq$ means stronger) iff $s\supseteq s'$, $A\supseteq A'$ and $\forall n \in s\setminus s'\, (n \notin \bigcup A')$.
\end{proof}

Notice that $\mathcal A$ may potentially have a property $P$ and potentially have property $\neg P$.

\section{Pseudonormality and \texorpdfstring{$\lambda$}{lambda}-sets}
In \cite{vinichim} the authors extended the definition of strongly $\aleph_0$-separated almost disjoint family, introduced in \cite{paulsergio}, in the following way: an almost disjoint family $\mathcal A$ is said to be strongly $(\aleph_0,<\!\mathfrak c)$-separated if and only if every pair of disjoint subfamilies of $\mathcal A$ can be separated, provided one is countable and the other has size less than $\mathfrak c$. This stronger property is useful to distinguish, at least consistently, almost-normal and strongly $\aleph_0$-separated almost disjoint families because, under the assumption $\mathfrak p>\omega_1$, a Luzin family is strongly $(\aleph_0,<\!\mathfrak c)$-separated and it is not almost normal. This separation property can be further extended in the following natural way:

\begin{defin}
We say that an almost disjoint family $\A$ is strongly $(\aleph_0,\mathfrak c)$-separated if and only if for every countable subfamily $\mathcal B\subset \A$, $\mathcal B$ and $\A\setminus \mathcal B$ can be separated.
\end{defin}

It turns out that this property is the combinatorical equivalent of the well known topological property, pseudonormality. Recall that a topological space $\mathcal X$ is pseudonormal if and only if for every pair of disjoint closed sets, with at least one of them is countable,  can be separated by disjoint open sets.

\begin{prop}
Let $\mathcal A$ be an almost disjoint family. The following are equivalent:
\begin{enumerate}
    \item $\A$ is strongly $(\aleph_0,\mathfrak c)$-separated.
    \item $\Psi(\A)$ is pseudonormal.
\end{enumerate}
\end{prop}
\begin{proof}

$(1)\implies(2)$: Let $F, G$ disjoint closed subsets of $\Psi(\A)$, where $F$ is countable. Then, there exists a partitioner $Z\subset \omega$ for $\A\cap F$ and $\A\setminus F$:
\[
F\cap \A = \{a \in \mathcal A: a\subset^* Z \} \qquad \text{ and } \qquad \A\setminus F = \{a\in \mathcal A: a\cap Z =^*\emptyset \}.
\]

Let $H = \{a\in \A: a\subset^* Z \}\cup Z$. $H$ is clopen. It follows that $C = H \cup \Big( (F\cap \omega)\setminus  H \Big) =  H \cup (F\cap \omega)$ is a clopen set that separates $F\cap \A$ and $\A\setminus F$.

It is straightforward to check that $C' = C\setminus (G\cap \omega)$ is a clopen set in $\Psi(\A)$ that separates $F$ and $G$. Hence, $\Psi(\A)$ is pseudonormal.

$(2)\implies (1)$: Given a countable subset $\mathcal B$ of $\mathcal A$, it suffices to show a partitioner for $\mathcal B$ and $\mathcal A\setminus \mathcal B$. Let $C$ be a clopen set separating these two closed sets. It is straightforward to show that $P=C\cap \omega$ works.
\end{proof}

We will stick with the term pseudonormality. The following folklore proposition is already known, as noted in \cite{OnQsets}. The proof can be obtained as a corollary of Proposition 2.2. and we write it for completeness.

\begin{prop}\label{prop: lambda iff (w,c)-sep}
Let $X\subset 2^\omega$. The following are equivalent:
\begin{enumerate}
    \item $X$ is a $\lambda$-set;
    \item $\Psi(\A_X)$ is pseudonormal.
\end{enumerate}
\end{prop}
\begin{proof}

$(1)\implies (2)$: By the previous proposition, it is enough to prove that $\mathcal A$ is strongly $(\aleph_0, \mathfrak c)$-separated. If $\mathcal B\subset \A_X$ is a countable subfamily, there exists a countable set $Y\subset X$ such that $\mathcal B = \mathcal A_Y$ and $\A_X\setminus \mathcal B = \mathcal A_{X\setminus Y}$. Since $Z$ and $X\setminus Z$ are $F_\sigma$ sets, by Proposition~\ref{prop: fsigma sse separa}, $\mathcal B$ and $\A_X\setminus \mathcal B$ are separated.

$(2)\implies (1)$: Let $Y\subset X$ a countable set. Then $\mathcal \A_Y$ and $\mathcal A_{X\setminus Y}$ are separated by disjoint open subsets $U\supseteq \mathcal \A_Y$  and $V\supseteq A_{X\setminus Y}$. One can check that $U$ and $V$ are clopen sets. Then, by Proposition~\ref{prop: fsigma sse separa}, $Y$ is a $G_\delta$ set in $X$.
\end{proof}

The reader may notice that follows from the previous proof that in $\Psi(\A_X)$, for a pair of disjoint closed subsets $F,G$, provided at least one of them is countable:
\begin{equation}
F \text{ and } G \text{ are separated by open sets} \iff
F \text{ and } G \text{ are separated by clopen sets}.
\end{equation}

It is worth noting that by \cite[Proposition 3.3]{paulsergio}, we do have MAD strongly $\aleph_0$-separated families. However, the same does not happen for pseudonormal families since, as it is known, in a MAD family we cannot separate an infinite countable set from its complement.

\begin{prop}
If $\A$ is a pseudonormal almost disjoint family, then $\A$ is not MAD.
\end{prop}

It is straightforward from that definition that strongly $(\aleph_0,<\mathfrak c)$-separated almost disjoint families are strongly $\aleph_0$-separated and these two definitions are the same under CH. Moreover, it was proved in \cite{paulsergio} that almost-normal almost disjoint families are strongly $\aleph_0$-separated. It was noticed in \cite{vinichim} that under the assumption that $\mathfrak p>\omega_1$ there exists a strongly $(\aleph_0,<\mathfrak c)$-separated almost disjoint family that is not almost-normal (in fact any Luzin family would have this property).

\begin{prop}[CH]
There exists a pseudonormal almost disjoint family that is not almost-normal.
\end{prop}
\begin{proof}
Notice that by Theorem~\ref{teo: sigma if and only if almost-normal} and Proposition~\ref{prop: lambda iff (w,c)-sep}, it suffices to show that there exist a set $X\subseteq 2^\omega$ such that $X$ is a $\lambda$-set that is not a $\sigma$-set. Such a set exists under CH  \cite{Mauldin1977rectangles}.
\end{proof}

Recall that in a MAD family $\mathcal A$, it is not possible to separate any countable subfamily $\mathcal B$ from $\mathcal A\setminus \mathcal B$. In particular, the existence of an almost-normal MAD family would give an example of almost-normal almost disjoint family that is not strongly $(\aleph_0, \mathfrak c)$-separated. This discussion can be summarized by the following diagram:

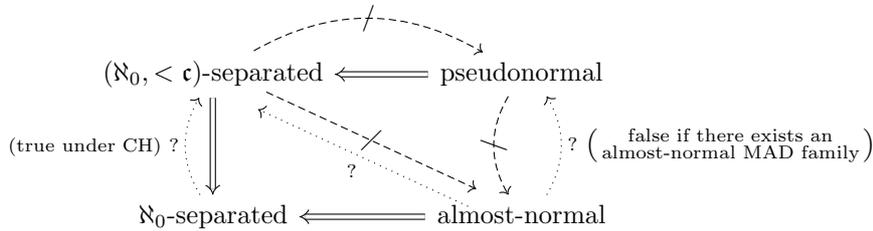
\begin{center}
% https://tikzcd.yichuanshen.de/#N4Igdg9gJgpgziAXAbVABwnAlgFyxMJZABgBoBGAXVJADcBDAGwFcYkQAKAHS6ZjQAWAfTI8+gkQEoeOGAA8cwALRx+AXxBrS6TLnyEU5UsWp0mrdt16N+wsgB4eAW3o4BAMwBO9ANYACAGNpLlkFZVU0DS0dbDwCIgAmClMGFjZEEBl5RSYnTBwlSE8XRijtEAxY-UTjFPN0zjEbCVEuFzcvX0Dg0MUVdU1TGCgAc3giUC8IJyQjEBwIJGJokCmZxCT5xcRl8rWkAGYaBdmV-cQ5k42aOAEsdxwkJTnGLDAGqHpb4c09z2mkJsrnNbvdHoglJtXu92FAIDhZFBfpN-usyFtZjRoR94YiQDQBDB6EjEGBmIxGMd6FhGOxIDCsfQAEYwRgABV0cQMIBsD3xIBZYBJB12KIBiCOGOuIFBfJ2WLeHy+hKRZ1RgOO20lsvBUMVsOVPzV4suWpoguF6OxBu+qsoaiAA
\begin{tikzcd}[column sep=1.2cm, row sep=1.3cm]
{(\aleph_0,<\mathfrak c)\text{-separated}} 
        \arrow[d, Rightarrow] 
        \arrow[rd, dashed, shift left, degil] 
        \arrow[r, dashed, bend left, degil]                                & 
\text{pseudonormal}
        \arrow[l, Rightarrow] 
        \arrow[d, dashed, degil, bend right]                                \\
{\aleph_0\text{-separated}} 
        \arrow[u, dotted, bend left, "\text{(true under CH) } ?"]                                   &  
\text{almost-normal} 
        \arrow[l, Rightarrow] 
        \arrow[lu, dotted, shift left=2, "?"] 
        \arrow[u, dotted, shift right=2, bend right, "? \,\, \big(\begin{smallmatrix} \text{false if there exists an} \\ \text{almost-normal MAD family} \end{smallmatrix}\big)"']
\end{tikzcd}
\end{center}

In the preceding diagram the double arrows are the results that holds in ZFC, the crossed dashed arrows are counter-examples that assumes additional combinatorial axioms and the dotted arrows are implications that remain unknown.

\begin{question}
In ZFC, is every strongly $\aleph_0$-separated almost disjoint family is strongly $(\aleph_0,<\!\mathfrak c)$-separated?
\end{question}

\begin{question}
Are almost-normal almost disjoint families strongly $(\aleph_0,<\!\mathfrak c)$-separated? Assuming additional axioms, can one construct an almost-normal almost disjoint family that is not strongly $(\aleph_0,<\!\mathfrak c)$-separated?
\end{question}

In the previous section we mentioned that in \cite{nluzin} it was proved that $\mathcal A$ is potentially normal iff $\mathcal A$ does not contain $n$-Luzin gaps. To prove this theorem, the authors of \cite{nluzin} used a forcing notion denoted by $\mathcal S_{\mathcal B, \mathcal C}$. Given an almost disjoint family $\mathcal A$ and disjoint subsets $\mathcal B, C$ of $\mathcal A$, $\mathcal S_{\mathcal B, \mathcal C}$ is the set of all triples $(s, \mathcal F, \mathcal G)$ such that $s \in 2^{<\omega}$, $\mathcal F \in [\mathcal B]^{<\omega}$, $\mathcal G \in [\mathcal C]^{<\omega}$ and $(\bigcup \mathcal F)\cap(\bigcup \mathcal G)\subseteq |s|$. We order $\mathcal S_{\mathcal B, \mathcal C}$ by letting $(s, \mathcal F, \mathcal G)\leq (s', \mathcal F', \mathcal G')$ iff $s'\subseteq s$, $\mathcal F'\subseteq \mathcal F$, $\mathcal G'\subseteq \mathcal G$, $\forall n \in |s|\setminus|s'|( n \in \bigcup \mathcal F'\rightarrow s(n)=1)$ and $\forall n \in |s|\setminus|s'|( n \in \bigcup \mathcal F'\rightarrow s(n)=0)$. By standard density arguments it is clear that if $G$ is a generic filter then $\bigcup\{s: (s, \mathcal F, \mathcal G) \in G\}$ is a characteristic function for a partitioner of $\mathcal B\cup \mathcal C$ separating $\mathcal B$ from $\mathcal C$.

In general this poset does not need to preserve cardinals, but it is c.c.c. (in fact, $\sigma$-centered) if either $|\mathcal B|=\omega$ or $|\mathcal C|=\omega$. So iterating this poset with standard bookkeping techniques we get the following result:

\begin{prop}
Every almost disjoint family is potentially strongly $(\aleph_0, \mathfrak c)$-separated. Moreover, we can show that by using a poset which does not increase the value of the continuum.
\end{prop}

As we have mentioned, this is not true for normality (e.g. Luzin families are not potentially normal, and we have provided an example of a Luzin$^*$-family which is not potentially almost-normal).

\section{Weak \texorpdfstring{$\lambda$}{lambda}-sets and strongly \texorpdfstring{$\aleph_0$}{aleph0}-separated almost disjoint families }
In this section we introduce the weak $\lambda$-sets, a weakening of the notion of $\lambda$-sets which relate to strongly $\aleph_0$-separated almost disjoint families in the same way as $Q$-sets relate to normal almost disjoint families. 

We consider $[\omega]^{\omega}$ with the (Polish) topology obtained by identifying $[\omega]^{\omega}$ with the subspace of $2^\omega$ corresponding to the characteristic functions of infinite sets. As mentioned in the introduction, if $\mathcal A$ is normal (pseudonormal) then $\mathcal A$ is a $Q$-set ($\lambda$-set) when viewed as a subspace of $[\omega]^\omega$. \cite{qsetsnormality}

\begin{defin}
We say that $X\subseteq 2^\omega$ is a weak $\lambda$-set iff for every pair of countable disjoint sets $Y,Z\subseteq X$ there exist a $G_\delta$-$F_\sigma$ set $H$ of $X$  such that $Y\subseteq H$ and $Z\subseteq X\setminus H$.
\end{defin}

Notice that given an almost disjoint family $\A$ and a partitioner $Z$ for $\A$, the sets $\mathcal B=\{a\in \mathcal A: a\subseteq^* Z\}$ and $\A\setminus\mathcal B =\{a\in \A: a\cap Z=^*\emptyset\}$ are relative $F_\sigma$'s and $G_\delta$'s of $A$ simultaneously.

\begin{prop}
If $\mathcal A$ is a strongly $\aleph_0$-separated almost disjoint family, then $\mathcal A$ is a weak $\lambda$-set of $\mathcal P(\omega)$.
\end{prop}
\begin{proof}
Given two uncountable disjoint families $\mathcal B, \mathcal C\subseteq \mathcal A$, by our assumption there exist a set $J\subseteq \omega$ such that:
\[
\mathcal B \subseteq \{a\in\mathcal A:a\subseteq^* J  \}\qquad \text{ and }\qquad \mathcal C\subseteq \{a\in\mathcal A: a\cap J=^*\emptyset\}.
\]

Observe that these sets can be rewritten as:
\begin{gather*} 
\bigcup\limits_{n\in\omega}\bigcap\limits_{m\in\omega} \{ a \in \mathcal A:  m\in a\setminus n \implies m\in  J\}\\ 
\bigcup\limits_{n\in\omega}\bigcap\limits_{m\in\omega} \{ a \in \mathcal A: m\in a\setminus n \implies m\notin J\}
\end{gather*}

Thus they are are $F_\sigma$ sets of $\mathcal P(\omega)$ containing $\mathcal B$ and $\mathcal C$, respectively. Since $J$ is a partitioner, the conclusion follows.
\end{proof}

\begin{prop}
Let $X\subseteq 2^\omega$. Then $X$ is a weak $\lambda$-set iff $\mathcal A_X$ is strongly $\aleph_0$-separated.
\end{prop}

\begin{proof}
First, suppose $X$ is a weak $\lambda$-set. Let $\mathcal B, \mathcal C \subset \A_X$ be countable disjoint sets. Let $Y, Z\subseteq X$ be (disjoint) sets such that $\mathcal B=\mathcal A_Y$ and $\mathcal C=\mathcal A_Z$. There exists a partitioner $P$ separating $X$ and $Y$, so by Proposition~\ref{prop: fsigma sse separa}, $\mathcal B$ and $\mathcal C$ are separated by a partition of $G_\delta$'s.

Let $Y, Z\subset X$ be disjoint countable sets. Then $\mathcal \A_Y$ and $\mathcal A_{Z}$ are separated by some partitioner. Then, again by Proposition~\ref{prop: fsigma sse separa}, $Y$ and $Z$ is separated by a partition of $G_\delta$'s.
\end{proof}

Of course, a $\lambda$-set is a weak $\lambda$-set. Moreover, we have the following:

\begin{teo}\label{teo: weak lambda is perf meager}
Let $X$ be a Polish space. Then every weak $\lambda$-set of $X$ is perfectly meager.
\end{teo}

\begin{proof}
Let $Z$ be a weak $\lambda$ subset of $X$. We can suppose that $Z$ is uncountable. Fix a perfect set $P$. Write $Z\cap P = Y\cup C$ where $Y=\{x \in Z\cap P: \text{for every open nhood } U\text{ of $x$, } |U\cap Z\cap P|\geq \omega_1\}$ and $C=Z\cap P\setminus Y$. Notice that $C$ is countable since the space is second countable, and that $Y$ is dense in itself and nonempty. It suffices to show that $Y$ is meager in $P$.

It is straightforward to construct two countable disjoint subsets $F,K\subset Y$ such that $\overline Y=\overline F=\overline K$.

Since $Z$ is a weak $\lambda$-set, there exist sequences of open sets $A_n$, $B_n$ ($n \in \omega$) such that:

\newpage
\begin{enumerate}
    \item $F\subset \bigcap\limits_{n\in\omega} A_n$;
    \item $K\subset \bigcap\limits_{n\in\omega} B_n$;
    \item $\bigcap\limits_{n\in\omega} A_n \cap Z \cap \bigcap\limits_{n\in\omega} B_n=\emptyset$;
    \item $Z\cap \bigcap\limits_{n\in\omega} B_n = Z\setminus \bigcap\limits_{n\in\omega} A_n$.
\end{enumerate}

For each $n\in\omega$, let $G_n = A_n\cup (X\setminus \overline Y)$. Then each $G_n$ is an open dense set, furthermore, $G_n\cap P$ is dense in $P$ because:
\[
\overline{G_n\cap P} =\underbrace{ \overline{A_n \cap P}}_{\supseteq \overline {F}=\overline{Y}} \cup \overline{P\setminus \overline Y} = P.
\]

Since $G\cap P= \bigcap\limits_{n\in\omega}G_n\cap P$ is a dense $G_\delta$ in $P$, it is comeager in $P$ and we have that:
\[
Y\cap G \subset Y\cap \bigcap_{n\in\omega}A_n \subset Y \setminus \bigcap_{n\in\omega} B_n = \bigcup_{n\in\omega} Y\setminus B_n.
\]

Notice that for each $n$, $\overline{Y\setminus B_n }$ has empty interior since $\overline{Y\setminus B_n}\subset \overline Y \setminus B_n$ and $B_n$ is dense in $\overline Y$. Thus $Y\cap G$ is meager in $P$. But also, $Y\setminus G=Y\setminus(P\cap G)$ is meager in $P$. Hence, $Y$ is meager in $P$.
\end{proof}

Recall that a set $X$ of reals is a $\lambda'$-set if for each countable $A$, $X\cup A$ is a $\lambda$-set. Analogously, we define a set to be a weak $\lambda'$-set if for every countable set $A$, $X\cup A$ is a weak $\lambda$ set. It follows easily that weak $\lambda'$-sets are weak $\lambda$-sets. Moreover, in the light of the Theorem~\ref{teo: weak lambda is perf meager}, we have the following diagram of implications between special sets of reals:

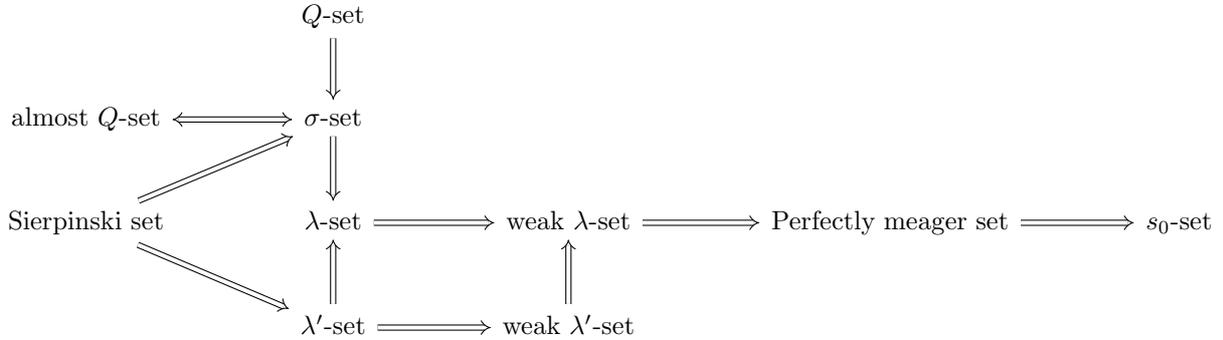
\begin{center}
    \begin{tikzcd}[column sep=1.5cm, row sep=0.8cm]
        & 
Q\text{-set}
        \arrow[d, Rightarrow]       &
                                    &
                                    &
                                    \\
\text{almost $Q$-set} 
        \arrow[r, Leftrightarrow]    &
\sigma\text{-set} 
        \arrow[d, Rightarrow]         &
                                    & 
                                    &
                                    \\
\text{Sierpinski set} 
        \arrow[ru, Rightarrow] 
        \arrow[rd, Rightarrow]      & 
\lambda\text{-set} 
        \arrow[r, Rightarrow]        &
\text{weak }\lambda\text{-set} 
        \arrow[r, Rightarrow]       &
\text{Perfectly meager set} 
        \arrow[r, Rightarrow]       & 
s_0\text{-set}                      \\
                                    &
\lambda'\text{-set} 
        \arrow[r, Rightarrow] 
        \arrow[u, Rightarrow]       & 
\text{weak }\lambda'\text{-set}
        \arrow[u, Rightarrow]       &
                                    &
\end{tikzcd}
\end{center}

In \cite{OnQsets}, forcing was used to construct a consistent example of a $Q$-set $X$ which is concentrated on a countable dense subset $F$ of $2^\omega$ ($F$ is in the ground model). The proof actually shows that $X$ is concentrated in every dense subset $F'$ of $F$ which is in the ground model. Therefore, by letting $F_0$, $F_1$ be two disjoint dense subsets of $F$ in the ground model, we get the set $Y=F_0\cup F_1\cup X$ is perfectly meager and not a weak $\lambda$-set (in the forcing extension). This later fact holds since it is easy to verify that a weak $\lambda$-set cannot be concentrated on two pairwise disjoint countable subsets. This discussion yields the following proposition:

\begin{prop} It is consistent that there exists a perfectly meager set which is not a weak $\lambda$-set and it is consistent that the class of weak $\lambda$-sets is not an ideal.
\end{prop}

It is still possible that every weak $\lambda$-set is a $\lambda$-set (in ZFC), but it is worth noting that the existence of weak $\lambda$-set that is not a $\lambda$-set cannot be weak $\lambda'$.

\begin{question}
Is there, at least consistently, a weak $\lambda$-set that is not a $\lambda$-set? In the negative case, is there a weak $\lambda$-set that is not weak $\lambda'$-set?
\end{question}

The next proposition is proved in section 9 of \cite{douwen1984the}.
\begin{prop}
$\mathfrak b$ is the least size of a non $\lambda$-set and the least size of a non $\sigma$-set.
\end{prop}

The following diagram describes the relations we have done so far for almost disjoint families of branches. The double arrows are the results that holds in ZFC, the dashed arrows are implications that assumes additional combinatorial axioms and the dotted arrows are implications that remain unknown. The first line stands for the least size of a set of reals which does not have the respective property. $\mathfrak q$ is defined simply as the least size of a set of reals which is not a $Q$-set. For more on $\mathfrak q$, see \cite{Brendle}. The fact that $\mathfrak b$ is the least size of a non $\lambda$-set and of a non $\sigma$-set is discussed in \cite{VanDouwen}. The least size of a non perfectly meager set is $\text{non}(\mathcal M)$ since the least size of a nonmeager set of a Polish space with no isolated points does not depend on the space.

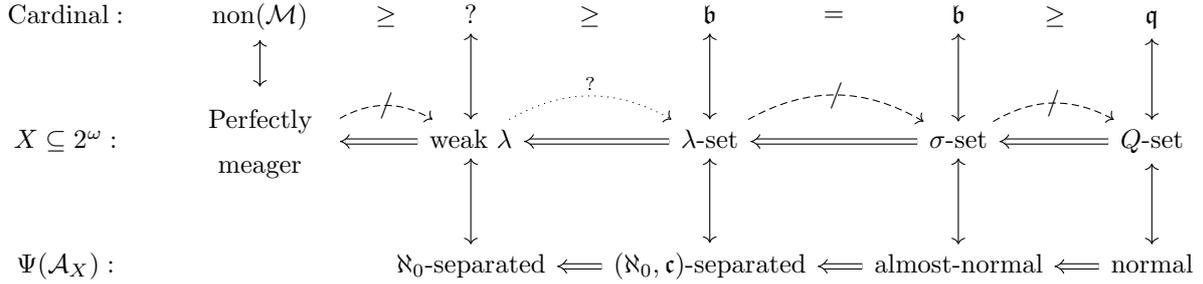
\begin{center}
% https://tikzcd.yichuanshen.de/#N4Igdg9gJgpgziAXAbVABwnAlgFyxMJZABgBoBGAXVJADcBDAGwFcYkQANAHS7mYCM4MHDACOAAgBMAPR4QAtjADm9RCAC+pdJlz5CKMpOp0mrdjwAK2ABQ959HAAsAxk3EBBAPocAlGs3a2HgEROQUxgwsbGpwnsQaWiAYQXpEkuE0kWZqPPzKWGDAOPT8zIz0AE7qwM7q4hYwFQBmMM44jACePHYw9EqNPDBgUEUlZZXqCYG6ISgAzBkmUeZcIgAeOMAA7r0A1uLqPOXy-FD0U0k6wfrIACyLWdEgR-QnZzzrmwC0QjiTAZcUrNkABWB6mJ48bBKewfGAbYA-YT-RLJGY3ABs4OWagAinCEUi-hc0dciFijJkIStPsBIBV7IwUdMySgwZSltlnqt4ZsmPJMDgvvTGczAeiiPcOY92LYuEwYGhHHFSHYHI4mhV6PtnD4Cd8hGhKg4YFAxaTUvNSNLqWpxHKFUqVTxHcriHqeYTDcaRGaSVdLch0sQIrbwAQ5fYnK5GOIALI+f1Am4LENUnEgAD8SYlKHuac5kK4UY1Wv2-BzrNBpALMpyxfVmu14grAItwKxtbDaqcTf2okrgbCXYzPdLzecg+BZBHXNy+UKxVK5SqNXUcEZjG6XFcFSgBSYg2Go2XEw0xlN-QQKFAmoUSHuIBwECQ5HIALv8iQYKfL8Q5GID8KnvRAsV-JAAE4gJAgB2GhnyQAAOaCv0QODwP-DEUKQMCEP-EFsMQH88PIW5CMfEi5kIhYMPISRCPSWiYMIuj4L-chmMST9Xxokj6K44DUNItieJY4j2LIgSQPIXD2IIqTUNk79CPQvCsIU78RMQSTb0Eh8tKojTEF4v9+N0kCTP0kA8mGJA5jIEBGAKJ4zjgRxTRAGhyjyRgLADWZHJgJocAubjtK0n8bKgOyHKcsAXIgHBfU8xyShgXz-P0QLgpSuss1CvSiK0sCopirznPYVz3OiwilLQmhSuM2KKrUKqPMIxCtKgoyIK0gCWIckj30odQgA
\begin{tikzcd}[column sep=0.6cm, row sep=normal]
\text{Cardinal} :                               &
\text{non}(\mathcal M)                                 & 
? 
        \arrow[l,phantom, "\geq" description]   & 
\mathfrak b 
        \arrow[l,phantom, "\geq" description]   & 
\mathfrak b 
\arrow[l,phantom, "=" description]      & 
\mathfrak q 
        \arrow[l,phantom, "\geq" description]   \\
X\subseteq 2^\omega:                    &
\begin{tabular}{c} Perfectly\\meager\end{tabular} 
        \arrow[u,leftrightarrow]
        \arrow[r,start anchor={[yshift=-0.3cm]}, dashed, bend left,degil]           & 
\text{weak }\lambda 
        \arrow[d,leftrightarrow] 
        \arrow[u,leftrightarrow] 
        \arrow[l,Rightarrow] 
        \arrow[r, "?", dotted, bend left]       &
\lambda\text{-set} 
        \arrow[d,leftrightarrow] 
        \arrow[u,leftrightarrow] 
        \arrow[l,Rightarrow] 
        \arrow[r, dashed, bend left,degil]      & 
\sigma\text{-set} 
        \arrow[d,leftrightarrow] 
        \arrow[u,leftrightarrow] 
        \arrow[l,Rightarrow] 
        \arrow[r, dashed, bend left,degil]      &
Q\text{-set} 
        \arrow[d,leftrightarrow] 
        \arrow[u,leftrightarrow] 
        \arrow[l,Rightarrow]                    \\
\Psi(\mathcal A_X):                     &
                                        &  
\aleph_0\text{-separated}         &
(\aleph_0,\mathfrak c)\text{-separated} 
        \arrow[l,Rightarrow]                    & 
\text{almost-normal} 
        \arrow[l,Rightarrow]                    & 
\text{normal} 
        \arrow[l,Rightarrow]
\end{tikzcd}
\end{center}

In view of the preceding results, we pose the following:

\begin{question}
Does there (consistently) exist a weak $\lambda$-set that is not a $\lambda$-set? 
\end{question}

\section{Further discussion and more questions:}
        As discussed in the previous sections, we know that if $\mathcal A$ is normal then $\mathcal A$ is a $Q$-set (as a subset of $\mathcal P(\omega)$). However, the converse is consistently not true, as proved in \cite{miller2002madQset} where a model with a MAD family which is a $Q$-set is constructed. Given the relations between $\sigma$-sets and almost-normality, $\lambda$-sets and pseudonormality and weak $\lambda$-sets and strongly $\aleph_0$-separatedness, it is natural to ask if similar results hold for these notions. As we have already discussed, in this sense, strongly $\aleph_0$-separated almost disjoint families are weak $\lambda$-sets and pseudonormal almost disjoint families are $\lambda$-sets.
        
        The example from \cite{miller2002madQset} mentioned above shows us that almost disjoint families which are $\lambda$-sets don't need to be pseudonormal as subsets of $\mathcal P(\omega)$ since MAD families are not pseudonormal. 
        
        However, the following two questions remain open:
        
\begin{question}
Is there a MAD almost-normal family?
\end{question}

\begin{question}
Is it true that every almost disjoint family which is a weak $\lambda$-set is also $\aleph_0$-separated?
\end{question}

Proposition \ref{weaklyseparatedluzin} grants the existence of an almost disjoint family of size $\omega_1$ which is not almost-normal. Thus, assuming $\mathfrak b>\omega_1$, we have:

\begin{cor}
$\mathfrak b>\omega_1$ implies the existence of an almost disjoint family $\mathcal A\subseteq [\omega]^{\omega}$ of size $\omega_1$ which is a $\sigma$-set but is not almost-normal. Moreover, the existence of such an almost disjoint family is consistent with CH.
\end{cor}

\begin{proof}
For the second part of the theorem, start with a model of $\mathfrak b>\mathfrak c$ and force with a $\sigma$-closed forcing poset which collapses $\mathfrak c$ onto $\omega_1$, such as $\text{Fn}(\omega_1, \mathfrak c, \mathfrak \omega_1)=\{s: s\subseteq \omega_1\times \mathfrak c \text{ is a countable partial function}\}$ ordered by reverse inclusion. Since this poset does not add reals and sequences with range in $V$, all the uncountable almost-disjoint families which are $\sigma$-sets but are not almost-normal will be preserved in $V[G]$.
\end{proof}

In some sense, there is not much we can do regarding this implication since there are models without $\sigma$-sets (\cite{MILLER1979233}, Theorem 22), so in this model ``every almost disjoint family which is a $\sigma$-set is almost-normal'' is trivially true. Still, we ask:

\begin{question}
What are the relations between $\mathcal A$ being a $\sigma$-set of $\mathcal P(\omega)$ and $\mathcal A$ being almost-normal? If there is a $\sigma$-set, is there an almost disjoint family which is a $\sigma$-set but is not almost-normal?
\end{question}

The diagram below summarizes the known implications and the open questions concerning these properties of $\Psi(A)$ and properties of $A$ as a subset of $[\omega]^\omega$.
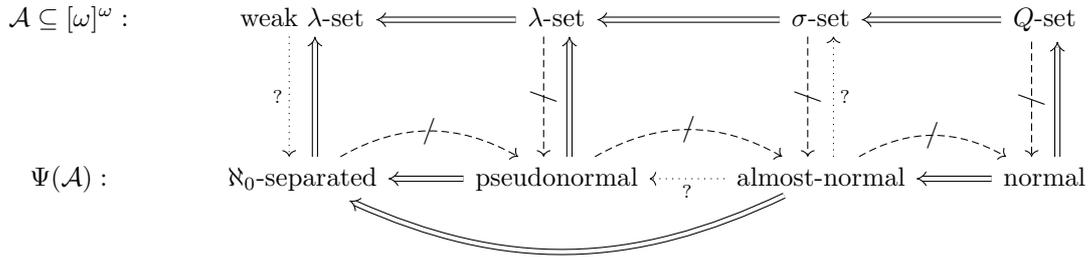
\begin{center}
    % https://tikzcd.yichuanshen.de/#N4Igdg9gJgpgziAXAbVABwnAlgFyxMJZABgBpiBdUkANwEMAbAVxiRAB12BbOnACwDGjAAQBBYZzhMARnBg4YAR2HJOELjADmdCgD01G7YhABfUuky58hFGQCMVWoxZtOABWwAKTj35CGYgCUxmYW2HgERHbkjvTMrMacCgAeOMAA7jB0ANYAtCacDHRc0lB0EuwpacJyOCam5iAY4dZEAEwx1HEuiexFJWVJMKnAubX1oU2WETbIAMydTvGu7NiaPEMjY-ITjc1WkSgALIvdCSAAiptp23UNYQez0Q5dzufe7IwwaHwA+mScL4-f6Ba6jORoAB0uweM3apBeSx6IA+QL+AO4vD4ADMAE45YQCUGVYY3CHQ+5TFqHeYI2JvFZVYCMXKQXE8BgwqmPIgnRFnRmk4BsjkTRwwKCaeBEUB49RIBYgHAQJBtSZyrhIE5KlWIObq3HyxAAVmoyq1BqNAE4zbrTSA4HwsNicEg7JbNYgbTqkAAOD1+21IADs1AYWDA5zKjolIDDdGkMAYbmmrWMDBgLspGpDQcQADYAwW8x0HU6s4h3Y0c4hQz69dRHc7XZWi7689rw5G2FAIDgFFA4yAionk6nDiBcVhNHxXa9lsYAPzZw2euvmxDtkCJsCDxBHb1dqN0GODtt5707ve5A9hiPH08ro35vN1q9IfNkYf3nsnvixosX3rLd30QXI5i-I9f0fc96yA0CDyLRUNzrJsK1LKDjGjf8z2rVckHtDdvTQltom-bssL-AC8KNUsNyAki3TvCiQF7fsAIoEwgA
\begin{tikzcd}[column sep=1cm, row sep=1.5cm]
{\mathcal A \subseteq [\omega]^\omega:}                             & 
\text{weak }\lambda\text{-set} 
        \arrow[d,"?"', dotted, shift right=2]                               &
\lambda\text{-set} 
        \arrow[l,Rightarrow] 
        \arrow[d, dashed, shift right=2,degil]                              &
\sigma\text{-set} 
        \arrow[l,Rightarrow]
        \arrow[d, degil, dashed, shift right=2, rightarrow]                  & 
Q\text{-set} 
        \arrow[l,Rightarrow] 
        \arrow[d, dashed, shift right=2, degil]                             \\
\Psi(\mathcal A):                                                   & 
{\aleph_0\text{-separated}} 
        \arrow[u, shift right=2,Rightarrow] 
        \arrow[r, dashed, bend left=30, degil]                      & 
{\text{pseudonormal}} 
        \arrow[l,Rightarrow] 
        \arrow[u, shift right=2,Rightarrow] 
        \arrow[r, dashed, bend left=30, degil]                              & 
\text{almost-normal} 
        \arrow[l, dotted, "?"] 
        \arrow[u, "?"', dotted, rightarrow, shift right=2] 
        \arrow[r, dashed, bend left=30, degil] 
        \arrow[ll, bend left=25,Rightarrow]                                 &
\text{normal} 
        \arrow[u, shift right=2,Rightarrow] 
        \arrow[l,Rightarrow]       
\end{tikzcd}
\end{center}

Looking in a different direction, one may ask about countable paracompactness-like properties of Isbell-Mrówka spaces. It is known that for Hausdorff spaces $X$, $X$ is countably paracompact iff for every decreasing sequence $(F_n: n \in \omega)$ of closed subsets of $X$ such that $\bigcap_{n \in \omega} F_n=\emptyset$ there exists open sets $(V_n: n \in \omega)$ such that $F_n\subseteq V_n$ and $\bigcap_{n \in \omega} \cl(V_n)=\emptyset$ (e.g. see \cite[Theorem 5.2.1]{engelking1977general}).

With the definition of almost-normality in mind, it is natural to define analogous weakenings of countable paracompactness associated to regular closed sets. For example,

\begin{defin}
We say a subset $F\subseteq X$ is $\sigma$-regular closed iff $F$ is a intersection of a countable family of regular closed sets.

We say a topological space $X$ is $\sigma$-almost countably paracompact ($\sigma$-almost countably paracompact) iff for every decreasing sequence $(F_n: n \in \omega)$ of regular closed ($\sigma$-regular closed) subsets of $X$ such that $\bigcap_{n \in \omega} F_n=\emptyset$, there exists open sets $(V_n: n \in \omega)$ such that $F_n\subseteq V_n$ and $\bigcap_{n \in \omega} \cl(V_n)=\emptyset$.
\end{defin}

$\Delta$-sets are special subsets of reals associated to countable paracompactness. We say $D\subseteq 2^\omega$ is a $\Delta$-set iff for each non-increasing sequence $(H_n)_{n\in \omega}$ of subsets of $D$ with empty intersection, there exists a sequence of open sets $(V_n)_{n\in \omega}$ of $D$, with empty intersection and $H_n\subseteq V_n$. Every $Q$-set is a $\Delta$-set, and for $X\subseteq 2^\omega$, $\Psi(\mathcal A_X)$ is countably paracompact iff $X$ is a $\Delta$-set. We refer to \cite[Section 8]{Reed} for more information in these sets.

Thus, it is natural to ask if there is a class of sets of reals characterizing $\sigma$-almost countably paracompact in the associated $\Psi$-space of branches. Somewhat surprisingly, any almost disjoint family of branches has this property:

\begin{prop}
For every $X\subseteq 2^\omega$, $\Psi({\mathcal A}_X)$ is $\sigma$-almost countably paracompact.
\end{prop}
\begin{proof}
Let $\mathcal A={\mathcal A}_X$ and $Y=\Psi(\mathcal A_X)$. Fix a decreasing sequence $(F_n: n \in \omega)$ of $\sigma$-regular closed subsets of $\Psi(\mathcal A)$. For each $i \in \omega$, let $(F_{i, n}: n \in \omega)$ be a family of regular closed sets such that $F_i=\bigcap_{n \in \omega} F_{i, n}$.

Let $U_i=Y\setminus F_i$ and $Z_i=\{x \in X: A_x\in U_i\}$. Notice that:

$$Z_i=\{x\in X:A_x \in Y\setminus\bigcap_{n \in \omega}F_{i, n}\}=\bigcup_{n \in \omega}\{x\in X:A_x \in Y\setminus F_{i, n}\}$$

By Proposition \ref{prop: fsigma sse separa}, each $Z_i$ is an $F_\sigma$, so there exists a sequence $(K_i: i \in \omega)$ of closed subsets of $X$ such that for each $i$, $K_i\subseteq Z_i$ and $\bigcup_{i \in \omega}K_i=\bigcup_{i \in \omega}Z_i=X$.

Write $2^{<\omega}=\{s_k: k \in \omega\}$. Let $\sigma:\omega\rightarrow \omega$ be strictly increasing such that for each $k \in \omega$, $s_k \in U_{\sigma(k)}$.

For each $i \in \omega$, let $L_i=\{A_x: x \in K_i\}\cup(\{x|_n: n \in \omega, \,x \in X\}\cap U_i)\cup\{s_k: \sigma(k)=i\}$.

It is now straightforward to verify that $\bigcup_{i \in \omega}\inter (L_i)=Y$, that $L_i\subseteq U_i$ and that $L_i$ is closed, so by letting $V_i=Y\setminus L_i$ for each $i \in \omega$ the proof is complete.
\end{proof}

\begin{question}
Is every Isbell-Mrówka space $\sigma$-almost countably paracompact?
\end{question}

\section{Acknowledgements}
The first author was funded by FAPESP (Fundação de Amparo à Pesquisa do Estado de São Paulo, process number 2017/15502-2). 

The second author was funded by CNPq (Conselho Nacional de Desenvolvimento Científico e Tecnológico, process number 141881/2017-8).

The third author acknowledges the support from NSERC.

The first author thanks Professor Artur Hideyuki Tomita for early discussions regarding this paper.

\bibliographystyle{plain}

\bibliography{bibliografia}

\end{document}